\documentclass{amsart}

\usepackage[latin1]{inputenc}
\usepackage[T1]{fontenc}
\usepackage{enumerate}

\newtheorem{theorem}{Theorem}[section]
\newtheorem{lemma}[theorem]{Lemma}

\theoremstyle{definition}

\newtheorem{example}[theorem]{Example}

\theoremstyle{remark}

\numberwithin{equation}{section}

\newcommand{\Z}{\mathbb{Z}}

\newcommand{\Span}{\text{span}}

\begin{document}

\title[Existence and non-existence of frequently hypercyclic subspaces]{Existence and non-existence of frequently hypercyclic subspaces for weighted shifts}

\author[Q. Menet]{Quentin Menet}
\address{Institut de Mathématique\\
Université de Mons\\
20 Place du Parc\\
7000 Mons, Belgique}
\email{Quentin.Menet@umons.ac.be}
\thanks{The author is supported by a grant of FRIA}


\date{}


\date{}

\maketitle

\begin{abstract}
We study the existence and the non-existence of frequently hypercyclic subspaces in Banach spaces. In particular, we give an example of a weighted shift on $l^p$ possessing a frequently hypercyclic subspace and an example of a frequently hypercyclic weighted shift on $l^p$ possessing a hypercyclic subspace but no frequently hypercyclic subspace. The latter example allows us to answer positively Problem~1 posed by Bonilla and Grosse-Erdmann in [Monatsh. Math. 168 (2012)].
\end{abstract}
\maketitle

\section{Introduction}

Let $X$ be a separable infinite-dimensional Banach space. We denote by $L(X)$ the space of continuous linear operators from $X$ to itself. 

An operator $T\in L(X)$ is said to be \emph{hypercyclic} if there exists a vector $x\in X$ (called hypercyclic) such that the set $\{T^nx:n\ge 0\}$ is dense in $X$. The hypercyclicity is an important notion of linear dynamics and we refer to the books~\cite{2Bayart2,2Grosse2} for more details about this theory. In 2004, Bayart and Grivaux~\cite{2Bayart0,2Bayart} defined a new notion of hypercyclicity, called frequent hypercyclicity. An operator $T\in L(X)$ is said to be \emph{frequently hypercyclic} if there exists a vector $x\in X$  (also called frequently hypercyclic) such that for any non-empty open set $U\subset X$,
\[\underline{\text{dens}}\{n\ge 0:T^nx\in U\}>0,\]
where for any subset $A$ of $\mathbb{Z}_+$ (the set of non-negative integers), we define the lower density of $A$ as 
\[\underline{\text{dens}}\, A=\liminf_{N\rightarrow \infty}\frac{\#(A\cap[0,N])}{N+1}.\]

A useful sufficient condition for frequent hypercyclicity has been given by Bayart and Grivaux~\cite{2Bayart0,2Bayart} and improved by Bonilla and Grosse-Erdmann~\cite{4Bonilla}:

\begin{theorem}[(Frequent Hypercyclicity Criterion)]
Let  $X$ be a separable infinite-dimensional Banach space and $T\in L(X)$. Suppose that there exist a dense subset $X_0$ of $X$ and a map $S:X_0\to X_0$ such that, for all $x\in X_0$,
\begin{enumerate}[\upshape (1)]
\item $\sum_{n=1}^{\infty}T^nx$ converges unconditionally,
\item $\sum_{n=1}^{\infty}S^nx$ converges unconditionally,
\item $TSx=x$.
\end{enumerate}
Then $T$ is frequently hypercyclic.
\end{theorem}

In this paper, we are interested in the existence of infinite-dimensional closed subspaces in which every non-zero vector is frequently hypercyclic. Such a subspace is called a \emph{frequently hypercyclic subspace}. The notion of a frequently hypercyclic subspace has been studied for the first time by Bonilla and Grosse-Erdmann~\cite{4Bonilla2}. In particular, they obtain the following sufficient condition for the existence of frequently hypercyclic subspaces:

\begin{theorem}[\cite{4Bonilla2}]\label{Crit M0}
Let $X$ be a separable infinite-dimensional Banach space and $T\in L(X)$. Suppose that
\begin{enumerate}[\upshape (1)]
\item $T$ satisfies the Frequent Hypercyclicity Criterion;
\item there exists an infinite-dimensional closed subspace $M_0$ of $X$ such that, for all $x\in M_0$, $T^{n}x\to 0$ as $n\to \infty$.
\end{enumerate}
Then $T$ has a frequently hypercyclic subspace.
\end{theorem}

The convergence to $0$ along the \emph{whole} sequence $(n)$ for each vector in $M_0$ is a strong assumption and the examples of frequently hypercyclic operators satisfying this criterion are mainly operators for which there exists some scalar $\lambda$ with $|\lambda|<1$ such that $\dim \ker (T-\lambda)=\infty$. In particular, so far no weighted shift with frequently hypercyclic subspaces is known. 


On the other hand, an operator $T$ possesses a \emph{hypercyclic subspace} if there exists an infinite-dimensional closed subspace in which every non-zero vector is hypercyclic for $T$. Obviously, if $T$ does not possess a hypercyclic subspace or if $T$ is not frequently hypercyclic, then $T$ does not possess a frequently hypercyclic subspace. Bonilla and Grosse-Erdmann have then asked the following question~\cite[Problem~1]{4Bonilla2}:\\

\noindent\emph{\textbf{Problem 1}: Does there exist a frequently hypercyclic operator that has a hypercyclic
subspace but not a frequently hypercyclic subspace?}\\


In Section~\ref{nonexist}, we state a sufficient condition for the non-existence of frequently hypercyclic subspaces. Then we show that there exists some frequently hypercyclic weighted shift on $l^p$ that possesses a hypercyclic subspace and that satisfies this condition. This allows us to answer positively Problem~1.

In Section~\ref{exist}, we state a slight improvement of Theorem~\ref{Crit M0} and we give an example of a weighted shift on $l^p$ possessing a frequently hypercyclic subspace.

%
%
%
%

\section{Non-existence of frequently hypercyclic subspaces}\label{nonexist}

In the case of hypercyclic subspaces, we know a simple criterion \cite{3Leon,Menet} such that if an operator $T$  satisfies this criterion, then $T$ does not possess any hypercyclic subspace. 

\begin{theorem}\label{CritE}
Let $X$ be a separable infinite-dimensional Banach space and ${T\in L(X)}$.
If there exists $C>1$, a closed subspace $E$ of finite codimension in $X$ and a positive integer $n\ge 1$ such that
\[\|T^nx\|\ge C\|x\| \quad \text{for any $x\in E$,}\]
then $T$ does not possess any hypercyclic subspace.
\end{theorem}

In fact, one can show that if $T$ satisfies the assumption of Theorem~\ref{CritE}, then each infinite-dimensional closed subspace $M$ of $X$ contains a vector $x$ such that ${\|T^n x\|\ge 1}$ for any $n$. The construction of such a vector $x$ relies on the following lemma.

\begin{lemma}[{\cite[Lemma~10.39]{2Grosse2}}]\label{basseq}
Let $X$ be a separable infinite-dimensional Banach space. For any finite-dimensional subspace $F$ of $X$, any $\varepsilon>0$, there exists a closed subspace $E$ of finite codimension in $X$ such that for any $x\in E$ and any $y\in F$,
\[\|x+y\|\ge \max \Big(\frac{\|x\|}{2+\varepsilon},\frac{\|y\|}{1+\varepsilon}\Big).\]
\end{lemma}

We would like a similar criterion for the notion of frequently hypercyclic subspaces. Obviously, if $T$ does not possess any hypercyclic subspace, then $T$ does not possess any frequently hypercyclic subspace. Nevertheless, in order to prove that an operator does not possess any frequently hypercyclic subspace, it is sufficient to prove that each infinite-dimensional closed subspace $M$ contains a vector $x$ such that the upper density of $\{n\ge 0:\|T^n x\|\ge C\}$ is equal to $1$ for some $C>0$, i.e.
\[\limsup_{N\to \infty} \frac{\#\{n\le N:\|T^n x\|\ge C\}}{N+1}=1.\]
Indeed, such a vector $x$ cannot be frequently hypercyclic since the lower density of $\{n\ge 0:\|T^nx\|<C\}$ would be equal to $0$.


\begin{theorem}\label{9no fhcsub}
Let $X$ be a separable infinite-dimensional Banach space and ${T\in L(X)}$.
If there exists $C>0$ such that for any $K\ge 1$, any infinite-dimensional closed subspace $M$ of $X$, any $x \in X$, there exists $x'\in M$ such that $\|x'\|\le \frac{1}{K}$ and
\[\sup_{k> K}\frac{\#\{n\le k:\|T^n (x+x')\|\ge C\}}{k+1}> 1-\frac{1}{K},\]
then $T$ does not possess any frequently hypercyclic subspace.
\end{theorem}
\begin{proof}
Let $\tilde{M}$ be an infinite-dimensional closed subspace of $X$. We show that there exists a vector $x\in \tilde{M}$ such that
\[\limsup_{N\to \infty} \frac{\#\big\{n\le N:\|T^n x\|\ge \frac{C}{2}\big\}}{N+1}=1.\]


To this end, we recursively construct a sequence $(x_n)_{n\ge 0}\subset \tilde{M}$ with $x_0=0$ and an increasing sequence $(k_n)_{n\ge 0}$ with $k_0=0$ such that for any $n\ge 1$,
\begin{enumerate}
\item $\|x_n\|\le \frac{1}{n^2}$;
\item we have \[\frac{\#\{j\le k_n:\|T^j(\sum_{k=1}^nx_k)\|\ge C\}}{k_n+1}> 1-\frac{1}{n^2};\]
\item for any $j\le k_{n-1}$, $T^jx_n\in \bigcap_{k\le n-1} E_{k,j}$, where $E_{k,j}$ is the closed subspace of finite codimension given by Lemma~\ref{basseq} for $F_{k,j}=\text{span}(T^jx_0,\dots, T^jx_k)$ and $\varepsilon=1$.
\end{enumerate}
At each step, we obtain such a vector $x_n$ by using our assumption for $K=\max\{n^2,k_{n-1}\}$, $M=\tilde{M}\cap \bigcap_{j\le k_{n-1}}T^{-j}(\bigcap_{k\le n-1} E_{k,j})$ and $x=\sum_{k=1}^{n-1}x_k$. 

We let $x:=\sum_{n=1}^{\infty} x_n$ and $J_n=\{j\le k_n:\|T^j(\sum_{k=1}^nx_k)\|\ge C\}$. We know that $x\in \tilde{M}$ and moreover, for any $n\ge 1$, any $j\in J_n$, we have
\begin{align*}
\|T^jx\|&=\Big\|\sum_{k=1}^{\infty}T^jx_k\Big\|\\
&\ge \frac{1}{2}\Big\|\sum_{k=1}^{n}T^jx_{k}\Big\| \quad \text{because }\sum_{k=n+1}^{\infty}T^jx_{k}\in E_{n,j} \text{ and }\sum_{k=1}^nT^jx_k\in F_{n,j} \\
&\ge \frac{C}{2} \qquad\qquad \quad\ \ \text{because }j\in J_n.\\
\end{align*}
We deduce that for any $n\ge 1$,
\[\frac{\#\big\{j\le k_n:\|T^j x\|\ge \frac{C}{2}\big\}}{k_n+1}\ge \frac{\# J_n}{k_n+1}\ge 1-\frac{1}{n^2} \xrightarrow[n\to \infty]{} 1.
\]
The result follows.
\end{proof}

In the case of weighted shifts, Theorem~\ref{9no fhcsub} gives us a simple sufficient condition for the non-existence of frequently hypercyclic subspaces. We recall that a weighted shift $B_w:l^p\to l^p$ is defined by
\[B_w((x_0,x_1,x_2,\dots))=(w_1x_1,w_2x_2,w_3x_3,\dots),\]
where $w=(w_n)_{n\ge 1}$ is a bounded sequence of non-zero scalars. In other words, we have $B_w e_k=w_{k}e_{k-1}$, where $(e_k)_{k\ge 0}$ is the canonical basis of $l^p$ and $e_{-1}=0$.

\begin{theorem}\label{9thm wsfh}
Let $B_w:l^p\to l^p$ be a weighted shift ($1\le p<\infty$).
If there exists a sequence $(C_k)_{k\ge 0}$ of positive numbers such that
$\sum_{k\ge 0}(1/C_k)^p<\infty$ and 
\begin{equation}
\frac{\#\{n\ge 0:\|B_w^ne_k\|\ge C_k\}}{k+1}\xrightarrow[k\to\infty]{} 1,
\label{9condc}
\end{equation}
then $B_w$ does not possess any frequently hypercyclic subspace.
\end{theorem}
\begin{proof}
Let $K\ge 1$, $M$ an infinite-dimensional closed subspace of $l^p$ and $x \in l^p$.
In view of Theorem~\ref{9no fhcsub}, it suffices to show that there exists $x'\in M$ such that $\|x'\|\le \frac{1}{K}$ and
\[\sup_{k> K}\frac{\#\{n\le k:\|B_w^n (x+x')\|\ge 1\}}{k+1}> 1-\frac{1}{K}.\]
We first remark that for any $x'\in l^p$, any $k\ge 0$,
\[\frac{\#\{n\le k:\|B_w^n(x+x')\|\ge 1\}}{k+1}\ge \frac{\#\{n\ge 0:\|(x_{k}+x'_{k})B_w^ne_{k}\|\ge 1\}}{k+1}\]
and thus, if $|x_{k}+x'_{k}|\ge 1/C_{k}$, we have
\[\frac{\#\{n\le k:\|B_w^n(x+x')\|\ge 1\}}{k+1}\ge \frac{\#\{n\ge 0:\|B_w^ne_{k}\|\ge C_{k}\}}{k+1}.\]
By \eqref{9condc}, we know that there exists $k_0$ such that for any $k\ge k_0$
\[\frac{\#\{n\ge 0:\|B_w^ne_k\|\ge C_k\}}{k+1}> 1-\frac{1}{K}.\]
We deduce that it suffices to find a vector $x'\in M$ such that $\|x'\|\le \frac{1}{K}$ and ${|x_{k}+x'_{k}|\ge 1/C_{k}}$ for some $k\ge \max(k_0, K+1)$. Let $k_1\ge\max(k_0,K+1)$ such that 
\begin{equation}
\Big(\sum_{k=k_1}^{\infty}|x_k|^p\Big)^{\frac{1}{p}}\le \frac{1}{2K}
\quad \text{and}\quad  \Big(\sum_{k=k_1}^{\infty}\frac{1}{C^p_k}\Big)^{1/p}\le \frac{1}{2K}.
\label{9ineqC}
\end{equation}
 Since $M$ is an infinite-dimensional closed subspace, we know that for any $k\ge 0$, $M$ contains a vector $y$ with valuation $v(y):=\inf\{j\ge 0:y_j\ne 0\}$ bigger than $k$. We can thus choose $x'\in M$ such that
\[v(x')\ge k_1 \quad\text{and}\quad \|x'\|= 1/K.\]
Using \eqref{9ineqC}, we deduce that
\begin{align*}
\Big(\sum_{k=k_1}^{\infty}|x'_k+x_k|^p\Big)^{1/p}&\ge \|x'\|-\Big(\sum_{k=k_1}^{\infty}|x_k|^p\Big)^{1/p}\\
&\ge \frac{1}{K}-\frac{1}{2K}=\frac{1}{2K}\\
&\ge \Big(\sum_{k=k_1}^{\infty}\frac{1}{C^p_k}\Big)^{1/p},
\end{align*}
and thus, there exists $k\ge k_1$ such that
\[|x'_{k}+x_{k}|\ge \frac{1}{C_{k}},\] 
which concludes the proof.
\end{proof}

We can now prove the existence of a frequently hypercyclic weighted shift possessing a hypercyclic subspace and no frequently hypercyclic subspace. Thanks to Bayart and Ruzsa~\cite{BayartR}, we know a characterization of frequently hypercyclic weighted shifts on $l^p$. A weighted shift $B_w:l^p\to l^p$ is frequently hypercyclic if and only if $B_w$ satisfies the Frequent Hypercyclicity Criterion, and if and only if
\begin{equation}
\sum_{k=1}^{\infty}\frac{1}{\prod_{\nu=1}^k|w_{\nu}|^p}<\infty.
\label{caracfhc}
\end{equation}
We also know a characterization of weighted shifts on $l^p$ with hypercyclic subspaces~\cite{Gonzalez,Menet}. A hypercyclic weighted shift $B_w:l^p\to l^p$ possesses a hypercyclic subspace if and only if
\begin{equation}
\sup_{n\ge 1}\inf_{k\ge 0}\prod_{\nu=1}^n |w_{k+\nu}|\le 1.
\label{carachcsub}
\end{equation}
In conclusion, it just remains to show that there is some bounded weighted sequence~$w$ satisfying \eqref{caracfhc}, \eqref{carachcsub} and the condition of Theorem~\ref{9thm wsfh}.

\begin{theorem}\label{mainthm}
There exists some frequently hypercyclic weighted shift on $l^p$ possessing a hypercyclic subspace and no frequently hypercyclic subspace.
\end{theorem}
\begin{proof}
Let $(a_n)_{n\ge 0}$ be a strictly increasing sequence of integers such that $a_0=1$. We consider the weighted sequence $w$ given by
$w_k=2$ if $k\in[a_{2n},a_{2n+1}[$ for some $n\ge 0$, and $w_k=1$ otherwise.

Since the sequence $w$ is bounded, the weighted shift $B_w$ is an operator on $l^p$ and in view of \eqref{caracfhc}, we know that if \[\sum_{n=0}^{\infty}\frac{a_{2n+2}-a_{2n+1}}{(2^{a_{2n+1}-a_{2n}})^p}<\infty,\] 
then $B_w$ is frequently hypercyclic because
\[\sum_{k=1}^{\infty}\frac{1}{\prod_{\nu=1}^{k}|w_{\nu}|^p}\le\sum_{n=1}^{\infty}\Big(\frac{1}{2^n}\Big)^p+
\sum_{n=0}^{\infty}\frac{a_{2n+2}-a_{2n+1}}{(2^{a_{2n+1}-a_{2n}})^p}.\]
Thanks to \eqref{carachcsub}, we also know that if $a_{2n+2}-a_{2n+1}$ tends to infinity, then $B_w$ possesses a hypercyclic subspace and finally, we remark that we have, for any $a_{2n+1}\le l< a_{2n+3}$, 
\[\#\{n\ge 0:\|B_w^ne_l\|\ge 2^{a_{2n+1}-a_{2n}}\}\ge a_{2n}.\]
Thus, if for any $a_{2n+1}\le l< a_{2n+3}$, we let $C_l=2^{a_{2n+1}-a_{2n}}$, then we have
\[\frac{\#\{n\ge 0:\|B_w^ne_l\|\ge C_l\}}{l+1}\ge \frac{a_{2n}}{a_{2n+3}} \quad
\text{and} \quad \sum_{l=a_1}^{\infty}\frac{1}{C^p_l}=\sum_{n=0}^{\infty}\frac{a_{2n+3}-a_{2n+1}}{(2^{a_{2n+1}-a_{2n}})^p}.\]
Therefore, if 
\[\sum_{n=0}^{\infty}\frac{a_{2n+3}-a_{2n+1}}{(2^{a_{2n+1}-a_{2n}})^p}<\infty \quad\text{and}\quad \frac{a_{2n}}{a_{2n+3}}\rightarrow 1,\]
we deduce from Theorem~\ref{9thm wsfh} that $B_w$ does not possess any frequently hypercyclic subspace.

We now show that if we let \[a_{2n}=1+n(n+1) \quad\text{and}\quad a_{2n+1}=1+(n+1)^2,\]
then each of the properties above is satisfied, which will conclude the proof. We first remark that
\[a_{2n+1}-a_{2n}=n+1=a_{2n+2}-a_{2n+1}.\] The operator $B_w$ is thus frequently hypercyclic because
 \[\sum_{n=0}^{\infty}\frac{a_{2n+2}-a_{2n+1}}{(2^{a_{2n+1}-a_{2n}})^p}=\sum_{n=0}^{\infty}\frac{n+1}{(2^{n+1})^p} <\infty,\] 
 and, as $a_{2n+2}-a_{2n+1}$ tends to infinity, the operator $B_w$ possesses a hypercyclic subspace. Moreover,
we remark that
 \[\frac{a_{2n}}{a_{2n+3}}=\frac{1+n(n+1)}{1+(n+2)^2}\rightarrow 1\]
and
\[\sum_{n=0}^{\infty}\frac{a_{2n+3}-a_{2n+1}}{(2^{a_{2n+1}-a_{2n}})^p}=\sum_{n=0}^{\infty}\frac{(n+2)^2-(n+1)^2}{(2^{n+1})^p}=\sum_{n=0}^{\infty}\frac{2n+3}{(2^{n+1})^p}<\infty.\]
The operator $B_w$ is thus a frequently hypercyclic weighted shift possessing a hypercyclic subspace and no frequently hypercyclic subspace.
\end{proof}

\section{Existence of frequently hypercyclic subspaces}\label{exist}

Let $A$ be a set of positive lower density and $(n_k)$ the increasing enumeration of~$A$.
It is obvious that if $J$ is a set of positive lower density, then $\{n_k:k\in J\}$ is still a set of positive lower density. Using this argument, it is not difficult to adapt the proof of Theorem~\ref{Crit M0} in order to obtain the following generalization:

\begin{theorem}\label{9fhcsubsi}
Let $X$ be a separable infinite-dimensional Banach space and ${T\in L(X)}$.
If $T$ satisfies the Frequent Hypercyclicity Criterion and if there exists an infinite-dimensional closed subspace $M_0\subset X$ and a set $A$ of positive lower density such that for any $x\in M_0$, \[T^nx\xrightarrow[n\rightarrow \infty]{n\in A} 0,\]
then $T$ possesses a frequently hypercyclic subspace.
\end{theorem}

We can now state a sufficient condition for a frequently hypercyclic weighted shift on $l^p$ to possess a frequently hypercyclic subspace.

\begin{theorem}\label{9fhcsubw}
Let $B_w:l^p\to l^p$ be a frequently hypercyclic weighted shift and
\[G(k,C):=\{n\ge 0:\|B_w^ne_k\|\le C\},\ k\ge 0,\ C>0.\]
If there exists a strictly increasing sequence $(k_l)_{l\ge 1}\subset \Z_+$  and $C>0$ such that $\bigcap_{l\ge 1}G(k_l,C)$ is of positive lower density, then $B_w$ possesses a frequently hypercyclic subspace.
\end{theorem}
\begin{proof}
Let $M_0=\overline{\Span}\{e_{k_l}:l\ge 1\}$. If $x\in M_0$ and $n\in \bigcap_{l\ge 1}G(k_l,C)$, we have, by definition of $G(k,C)$, 
\[\|B^n_wx\|^p=\sum_{l=1}^{\infty}|x_{k_l}|^p \|B^n_w{e_{k_l}}\|^p
=\sum_{k_l\ge n} |x_{k_l}|^p \|B^n_w{e_{k_l}}\|^p
\le \sum_{k_l\ge n} C^p|x_{k_l}|^p\xrightarrow[n\to \infty]{} 0.\]
Since each frequently hypercyclic weighted shift $B_w$ on $l^p$ satisfies the Frequent Hypercyclicity Criterion and since $\bigcap_{l\ge 1}G(k_l,C)$ is of positive lower density, we conclude by using Theorem~\ref{9fhcsubsi}.
\end{proof}

\begin{example}[\bf{Weighted shift with a frequently hypercyclic subspace}]\label{example}
Let $w=(4,1^{a_1},4,1^{a_2},4,1^{a_3},4,1^{a_{4}},4,\dots)$ where $(a_n)$ is an increasing sequence of positive integers and $1^{a_n}$ designates a block of $1$ of length $a_n$. For any value of $a_n$, the weighted shift $B_w$ is an operator on $l^p$, and by~\eqref{caracfhc}, $B_w$ is frequently hypercyclic if and only if
\begin{equation}
\sum_{k=1}^{\infty}\frac{a_k}{(4^k)^p}<\infty.
\label{9condfhcs}
\end{equation}
Moreover, we see that 
\begin{equation*}
G\Big(\sum_{k=1}^na_k+n,1\Big)=[0,a_n]\cup \Big[\sum_{k=1}^na_k+(n+1),\infty\Big[,
\end{equation*}
and since $(a_n)$ is an increasing sequence, we get
\begin{align*}
A:=\bigcap_{n=1}^{\infty}G\Big(\sum_{k=1}^na_k+n,1\Big)&=\bigcap_{n=1}^{\infty}\Big([0,a_n]\cup \Big[\sum_{k=1}^na_k+(n+1),\infty\Big[\Big)\\
&=[0,a_1]\cup\bigcup_{n=1}^{\infty} \Big[\sum_{k=1}^{n}a_k+(n+1),a_{n+1}\Big].
\end{align*}
In particular, we deduce that if
\begin{equation}
\liminf_{n}\frac{a_n-\sum_{k=1}^{n-1}a_k-n}{\sum_{k=1}^{n}a_k+(n+1)}>0,
\label{9lpd}
\end{equation}
then $A$ is a set of positive lower density.
We remark that, if we let $a_n=3^{n-1}$, then \eqref{9condfhcs} and \eqref{9lpd} are satisfied. Indeed, we have
\begin{gather*}\sum_{k=1}^{\infty}\frac{3^{k-1}}{(4^k)^p}<\infty \quad\text{and}\\
\liminf_{n}\frac{3^{n-1}-\sum_{k=1}^{n-1}3^{k-1}-n}{\sum_{k=1}^{n}3^{k-1}+(n+1)}=
\liminf_{n}\frac{3^{n-1}-\frac{1}{2} (3^{n-1}-1)-n}{\frac{1}{2}(3^{n}-1)+n+1}=\frac{1}{3}.
\end{gather*}
Therefore, we deduce from Theorem~\ref{9fhcsubw} that if $a_n=3^{n-1}$, then $B_w$ possesses a frequently hypercyclic subspace. \hfill\ensuremath{\square}
\end{example}

In view of Theorem~\ref{mainthm} and Example~\ref{example}, we deduce on the one hand that there exist some frequently hypercyclic weighted shifts on $l^p$ with a frequently hypercyclic subspace and some others without a frequently hypercyclic subspace, and, on the other hand, that the existence of a frequently hypercyclic subspace is not equivalent to the existence of a hypercyclic subspace and the frequent hypercyclicity. Therefore, we pose the following question:\\

\noindent\textbf{Question 1}: Which weighted shifts on $l^p$ possess a frequently hypercyclic subspace?

\end{document}